\documentclass[10pt]{article}

\usepackage[margin=1.45in]{geometry}
\usepackage{amsmath,amsthm,amssymb,tikz,bbm,float,titlesec,bm}

\theoremstyle{definition}
\newtheorem*{ack}{Acknowledgments}
\newtheorem{example}{Example}[section]

\newtheorem{note}[example]{Note}
\newtheorem{question}[example]{Question}
\newtheorem{conjecture}[example]{Conjecture}
\newtheorem{theorem}[example]{Theorem}
\newtheorem{lemma}[example]{Lemma}
\newtheorem{proposition}[example]{Proposition}
\newtheorem{corollary}[example]{Corollary}

\titleformat{\section}
  {\normalfont\fontsize{14}{15}\bfseries}{\thesection}{1em}{}

\titleformat{\subsection}
  {\normalfont\fontsize{11}{15}\bfseries}{\thesubsection}{1em}{}

\title{\Large Norm, trace, and formal codegrees \\ of fusion categories}
\author{\normalsize Andrew Schopieray\thanks{This material is based upon work supported by the National Science Foundation under Grant No. DMS-1440140, while the author was in residence at the Mathematical Sciences Research Institute in Berkeley, California, during the Spring 2020 semester.}}
\date{\today}

\begin{document}

\maketitle

\begin{abstract}
We prove several results in the theory of fusion categories using the product (norm) and sum (trace) of Galois conjugates of formal codegrees.  First, we prove that finitely-many fusion categories exist up to equivalence whose global dimension has a fixed norm.  Furthermore, with two exceptions, all formal codegrees of spherical fusion categories with square-free norm are rational integers.  This implies, with three exceptions, that every spherical braided fusion category whose global dimension has prime norm is pointed.  The reason exceptions occur is related to the classical Schur-Siegel-Smyth problem of describing totally positive algebraic integers of small absolute trace. 
\end{abstract}

\section{Introduction}

Fusion categories and their many variants (tensor, spherical, braided, modular, etc.) are a vast generalization of the representation theory of finite groups and finite-dimensional Hopf algebras.  The formal codegrees of a fusion category $\mathcal{C}$, a finite collection of numerical invariants associated to representations of the underlying Grothendieck ring defined in Section \ref{formalcodegrees}, have proven to be critical to this study and include the Frobenius-Perron dimension ($\mathrm{FPdim}(\mathcal{C})$) and global dimension ($\dim(\mathcal{C})$) for spherical fusion categories, as examples.  Formal codegrees are restrictive from a number-theoretic perspective because they are examples of totally positive cyclotomic integers, and the less-familiar algebraic $d$-numbers \cite[Definition 1.1]{codegrees}.  Frobenius-Schur indicators of semisimple quasi-Hopf algebras \cite{linchenko2000frobenius} and spherical fusion categories \cite{NG200734}, and higher Gauss sums of modular tensor categories \cite{schopieraywang} are other examples of algebraic $d$-numbers.  The main goal of this paper is to expand the general theory of formal codegrees of fusion categories which is mainly contained in \cite{codegrees,ost15,ostrikremarks} thus far.

\par For a formal codegree $f\in\mathbb{C}$ of a fusion category, let $N(f)$ be the norm of $f$, or the product of its Galois conjugates, and $\mathrm{Tr}(f)$ be the trace of $f$, or their sum.  Theorem \ref{normfiniteness} states that for each $m\in\mathbb{Z}_{\geq1}$ there exist finitely-many fusion categories $\mathcal{C}$ up to equivalence with $N(\dim(\mathcal{C}))=m$.  This \emph{norm finiteness} is also true for $\mathrm{FPdim}(\mathcal{C})$ but because the set of all Frobenius-Perron dimensions of fusion categories is a discrete subset of the positive real numbers \cite[Corollary 3.13]{paul}.  In Section \ref{divisors} we observe a unique feature of algebraic $d$-numbers: divisibility of algebraic $d$-numbers is equivalent to divisibility of their norms.  As a result (Theorem \ref{newprop}), if $f$ is a formal codegree of a fusion category $\mathcal{C}$ which  is not divisible by any rational integer, then the formal codegrees of $\mathcal{C}$ are precisely the Galois orbit of $f$.  There is only one family of this type known: the categories $\mathcal{C}(\mathfrak{sl}_2,\kappa-2)_\mathrm{ad}$ for prime $\kappa\in\mathbb{Z}_{\geq3}$ coming from the representation theory of $\mathcal{U}_q(\mathfrak{sl}_2)$ with $q$ a root of unity \cite{MR4079742}. The last result of this section, Theorem \ref{prop6}, pertains to spherical fusion categories.  Sphericality is a very weak assumption as all known examples of fusion categories possess a spherical structure.   We prove that if $\mathcal{C}$ is a spherical fusion category with a formal codegree $f$ of square-free norm, then $f\in\mathbb{Z}$ or $f=(1/2)(5\pm\sqrt{5})$.  A spherical fusion category posessing either $(1/2)(5\pm\sqrt{5})$ as a formal codegree is equivalent to one of the rank 2 spherical fusion categories $\mathrm{Fib}$ or its Galois conjugate $\mathrm{Fib}^\sigma$ \cite{ostrik}.

\par We then apply our general theory to classification results for spherical fusion categories in Section \ref{globp}.  Recently in \cite[Question 3.8]{you} it was asked whether a spherical braided fusion category of prime global dimension is pointed (all simple objects have Frobenius-Perron dimension $1$) or equivalent to $\mathrm{Fib}\boxtimes\mathrm{Fib}^\sigma$.  Theorem \ref{biggun} answers this question in the affirmative and generalizes the result to include arbitrary number fields.  That is to say any spherical braided fusion category whose global dimension has prime \emph{norm} is pointed, with the exception of $\mathrm{Fib}$, $\mathrm{Fib}^\sigma$ and $\mathrm{Fib}\boxtimes\mathrm{Fib}^\sigma$.  The reason these exceptional (not pointed) spherical braided fusion categories can occur is that the categorical dimensions of their simple objects are exceptional cases of a classical result of Cassels \cite[Lemma 3]{cassels} on cyclotomic integers $\alpha$ such that the absolute trace (average of Galois conjugates) of $|\alpha|^2$ is less than 2.  In Theorem \ref{previous}, we remove the assumption of a braiding for spherical fusion categories with global dimension whose norm is a \emph{safe prime}.  Safe primes $p\in\mathbb{Z}_{\geq2}$ are of the form $p=2q+1$ where $q\in\mathbb{Z}_{\geq2}$ is also prime.  It is expected, but currently not proven, that infinitely many safe primes exist.  The reason the exceptional (not pointed) spherical fusion categories can occur in this case is that all of their formal codegrees are of the form $p\cdot u$ where $u$ is one of the three totally positive algebraic integers of smallest absolute trace: $1$ or $(1/2)(3\pm\sqrt{5})$.  For any $\lambda<2$, the proof of finiteness, and classification of totally positive algebraic integers of absolute trace strictly less than $\lambda$ is known as the Schur-Siegel-Smyth trace problem.  We encourage the reader to refer to \cite{MR2428512} for an expository look at the long history of this problem which continues to this day.  The first result in this direction was due to Schur \cite[Satz VIII]{schur1918verteilung} who solved the finiteness problem for $\lambda=\sqrt{e}$ and the classification was subsequently solved by Siegel \cite[Theorem III]{siegel} for $\lambda=3/2$, which is sufficient for our purposes.  Lastly, we include a proof of Theorem \ref{previous} in the case $p=13$ (Example \ref{thurteen}) as a proof-of-concept that the methods we have developed for spherical fusion categories can be pushed outside of the safe primes, but certainly more tools will be needed to complete this classification for arbitrary fusion categories.

\par This is not the first time the Schur-Siegel-Smyth trace problem, nor the results on absolute trace of Cassels have appeared in the literature in relation to fusion categories.  Gelaki, Naidu, and Nikshych used Siegel's initial trace bound to study the number of zeroes in the $S$-matrix of a weakly integral modular tensor category \cite[Proposition 6.2]{MR2587410}.  Their result can be viewed as an extension of similar results for zeroes of characters of finite groups (refer to \cite{MR2599088} and references within).  Later, Calegari, Morrison, and Snyder \cite{MR2786219} improved upon Cassels result \cite[Lemma 3]{cassels} as a means to classify the smallest possible Frobenius-Perron dimensions of objects in fusion categories; Calegari and Guo would make further improvements in \cite{MR3814339}.  The results we prove here on formal codegrees of fusion categories continue this tradition, and further illuminate the necessity for algebraic number theory in the study of fusion categories.

\begin{figure}[H]
\centering
\begin{equation*}
\begin{array}{|c|c|c|c|}
\hline \text{Notation} & \text{Meaning} & \text{Notation} & \text{Meaning} \\\hline\hline
\alpha &  \text{algebraic integer} & \zeta_n &  \exp(2\pi i/n)  \\\hline
\mathbb{K},\mathbb{L},\ldots &  \text{algebraic number fields} &\mathbb{Q}(\zeta_n)^+ &  \mathbb{Q}(\zeta_n+\zeta_n^{-1})\\\hline
\mathbb{K}^\times,\mathbb{L}^\times,\ldots & \text{unit groups} & [n]\in\mathbb{Q}(\zeta_m) & \sin(n\pi/m)/\sin(\pi/m) \\\hline
[\mathbb{L}:\mathbb{K}] &  \dim_\mathbb{K}(\mathbb{L})  &  N(\alpha) & \prod_{\sigma\in\mathrm{Gal}(\mathbb{Q}(\alpha)/\mathbb{Q})}\sigma(\alpha) \\\hline
d_\alpha & [\mathbb{Q}(\alpha):\mathbb{Q}] & \mathrm{Tr}(\alpha) & \sum_{\sigma\in\mathrm{Gal}(\mathbb{Q}(\alpha)/\mathbb{Q})}\sigma(\alpha)  \\\hline
\end{array}
\end{equation*}
    \caption{Recurring notation}%
    \label{fig:X}%
\end{figure}


\section{Preliminaries}

This exposition is as self-contained as one could hope; we include the following preliminary sections with references for readers less familiar with either elementary number theory or fusion categories.  We also include a section about our main technical tool, formal codegrees, where the interplay of number theory and fusion categories is well illustrated.


\subsection{Number fields, norm, trace, and $d$-numbers}

What is needed here from algebraic number theory can be found in any undergraduate textbook on the subject (e.g. \!\cite{MR2031707}).  The following notation and language will be fixed throughout.  The field of complex numbers, $\mathbb{C}$, contains all other fields considered.  Let $\mathbb{Q}\subset\mathbb{C}$ be the field of rational numbers with algebraic closure $\overline{\mathbb{Q}}$.  By a \emph{number field} we will mean any extension $\mathbb{K}/\mathbb{Q}$ with $[\mathbb{K}:\mathbb{Q}]:=\dim_\mathbb{Q}(\mathbb{K})<\infty$.  An \emph{algebraic integer} is any $\alpha\in\mathbb{C}$ which is a root of a monic polynomial with coefficients in $\mathbb{Z}$ (the rational integers).  For brevity, if $\alpha$ is an algebraic integer, we set $d_\alpha:=[\mathbb{Q}(\alpha):\mathbb{Q}]$.  The set of all algebraic integers has the structure of a unital ring, $\mathbb{A}$, and the group of invertible elements, or \emph{algebraic units}, will be denoted $\mathbb{A}^\times$.  More generally, if $\mathbb{K}$ is any number field, then the ring of algebraic integers and the group of algebraic units contained in $\mathbb{K}$ will be denoted $\mathcal{O}_\mathbb{K}$ and $\mathcal{O}_\mathbb{K}^\times$, respectively.  Let $\alpha,\beta\in\mathbb{A}$, $\alpha\neq0$.  Then $\alpha$ \emph{divides} $\beta$ if and only if there exists $\gamma\in\mathbb{A}$ such that $\alpha\gamma=\beta$. 

\vspace{2 mm}

All number fields of interest to us will be \emph{cyclotomic} in light of \cite[Theorem 8.51]{ENO}, which is to say contained in $\mathbb{Q}(\zeta_n)$ where $\zeta_n:=\exp(2\pi i/n)$ for some $n\in\mathbb{Z}_{\geq2}$.  We refer the reader to \cite{washington} for further resources about cyclotomic number fields.  One benefit of operating in cyclotomic number fields is that they are all Galois extensions of $\mathbb{Q}$, i.e. \!number fields $\mathbb{K}$ with $\mathbb{K}/\mathbb{Q}$ a normal extension.  In particular $\mathrm{Gal}(\mathbb{Q}(\zeta_n)/\mathbb{Q}))\cong(\mathbb{Z}/n\mathbb{Z})^\times$.  The maximal totally real subfield of $\mathbb{Q}(\zeta_n)$ is $\mathbb{Q}(\zeta_n)^+:=\mathbb{Q}(\zeta_n+\zeta_n^{-1})$ with $[\mathbb{Q}(\zeta_n):\mathbb{Q}(\zeta_n)^+]=2$ for $n\neq2$.  Algebraic units of interest to us will be \emph{cyclotomic units}: those products of a root of $1$ and units of the form $(\zeta^n-1)/(\zeta-1)$ where $n\in\mathbb{Z}_{\geq1}$ and $\zeta$ is any root of $1$.  For example, let $p$ be an odd prime.  Lemma 8.1 of \cite{washington} describes multiplicative generators of the cyclotomic units in terms of the real units $[a]:=\sin(a\pi/p)/\sin(\pi/p)$ and $\zeta_p$.  In particular, the cyclotomic units of $\mathbb{Q}(\zeta_p)^+$ are multiplicatively generated by $\pm1$ and $A:=\{[a]:2\leq a\leq(p-1)/2\}$.

\vspace{2 mm}

We will often refer to the product and sum of all Galois conjugates of algebraic integers.  If $\mathbb{K}$ is a cyclotomic number field (more generally, if $\mathbb{K}/\mathbb{Q}$ is Galois) we may safely define the \emph{norm} function $N_\mathbb{K}:\mathbb{K}\to\mathbb{Q}$ by
\begin{equation}
N_\mathbb{K}(\alpha):=\prod_{\tau\in\mathrm{Gal}(\mathbb{K}/\mathbb{Q})}\tau(\alpha).
\end{equation}
When there is no risk of ambiguity, as will almost always be the case, we will reduce this notation to $N(\alpha):=N_{\mathbb{Q}(\alpha)}(\alpha)$.  The norm function is multiplicative and when $\alpha\in\mathcal{O}_\mathbb{K}$, $N_\mathbb{K}(\alpha)\in\mathbb{Z}$.  Therefore if $\mathbb{Q}(\alpha)\subset\mathbb{K}$, $N_\mathbb{K}(\alpha)=N(\alpha)^{[\mathbb{K}:\mathbb{Q}(\alpha)]}$.  Algebraic units are those $u\in\mathcal{O}_\mathbb{K}$ with $N_\mathbb{K}(u)=\pm1$.  Likewise, we define the \emph{trace} function $\mathrm{Tr}_\mathbb{K}:\mathbb{K}\to\mathbb{Q}$ by
\begin{equation}
\mathrm{Tr}_\mathbb{K}(\alpha):=\sum_{\tau\in\mathrm{Gal}(\mathbb{K}/\mathbb{Q})}\tau(\alpha)
\end{equation}
and $\mathrm{Tr}(\alpha):=\mathrm{Tr}_{\mathbb{Q}(\alpha)}(\alpha)$.  The trace function is additive and when $\alpha\in\mathcal{O}_\mathbb{K}$, $\mathrm{Tr}_\mathbb{K}(\alpha)\in\mathbb{Z}$.  Therefore if $\mathbb{Q}(\alpha)\subset\mathbb{K}$, then $\mathrm{Tr}_\mathbb{K}(\alpha)=[\mathbb{K}:\mathbb{Q}(\alpha)]\mathrm{Tr}(\alpha)$.  If $x^n+a_1x^{n-1}+\cdots+a_{n-1}x+a_n$ is the minimal polynomial of a cyclotomic integer $\alpha$, then $\mathrm{Tr}(\alpha)=-a_1$ and $N(\alpha)=(-1)^na_n$.

\vspace{2 mm}

Lastly, a very specific type of algebraic integer appears abundantly in the study of fusion categories, introduced by Ostrik in \cite{codegrees}.  An algebraic \emph{$d$-number} is a nonzero $\alpha\in\mathbb{A}$ such that $\sigma(\alpha)/\alpha\in\mathbb{A}$ for all $\sigma$ in the absolute Galois group $\mathrm{Gal}(\overline{\mathbb{Q}}/\mathbb{Q})$.  Lemma 2.7 of \cite{codegrees} gives many alternative characterizations of a $d$-number.  For our purposes the most important characterization is that a nonzero $\alpha\in\mathbb{A}$ is a $d$-number if and only if the minimal polynomial of $\alpha$ over $\mathbb{Q}$ (which lies in $\mathbb{Z}[x]$), $x^n+a_1x^{n-1}+\cdots+a_{n-1}x+a_n$, enjoys the property that $(a_n)^j$ divides $(a_j)^n$ for all $1\leq j\leq n$.  If $\alpha$ is a cyclotomic $d$-number, then by the proof of \cite[Lemma 2.7]{codegrees},
\begin{equation}\label{ddd}
\alpha^{d_\alpha}=N(\alpha)\cdot u=(-1)^na_n\cdot u
\end{equation}
for some $u\in\mathcal{O}_{\mathbb{Q}(\alpha)}^\times$.  This fact is crucial for future arguments.


\subsection{Fusion categories}

\par In what follows we will consider fusion categories over $\mathbb{C}$ in the sense of \cite{tcat,ENO}.  These are semisimple tensor categories \cite[Definition 4.1.1]{tcat} over $\mathbb{C}$ whose set of isomorphism classes of simple objects, $\mathcal{O}(\mathcal{C})$, is finite.  One important property of tensor categories is the existence of \emph{duality} of objects: a permutation of $\mathcal{O}(\mathcal{C})$ which we denote by $X\mapsto X^\ast$ (without the assumption of semisimplicity, one would need to consider \emph{left} and \emph{right} duals separately \cite[Proposition 4.8.1]{tcat}).  The decomposition of $\otimes$-products into elements of $\mathcal{O}(\mathcal{C})$ (afforded by semisimplicity) are known as the \emph{fusion rules} of $\mathcal{C}$.

\begin{example}
The canonical examples of fusion categories are $\mathrm{Rep}(G)$ and $\mathrm{Vec}_G$, the categories of finite-dimensional complex representations of a finite group $G$, and finite dimensional complex $G$-graded vector spaces \cite[Examples 2.3.4--2.3.6]{tcat}.  When $G$ is the trivial group we recover the trivial fusion category $\mathrm{Vec}$, the unique fusion category $\mathcal{C}$ (up to equivalence) with $|\mathcal{O}(\mathcal{C})|=1$.
\end{example}

The fusion rules of a fusion category $\mathcal{C}$ are a property of the underlying Grothendieck ring, $K(\mathcal{C})$, hence a single Grothendieck ring can have many categorifications, corresponding to various associativity isomorphisms for these products.  For example the group ring $\mathbb{Z}[G]$ is categorified by $\mathrm{Vec}_G$, but the associativity may also be twisted by a $3$-cocycle $\omega$ to produce potentially inequivalent categories $\mathrm{Vec}_G^\omega$ \cite[Example 2.3.8]{tcat}.  One can consider these associativity isomorphisms as solutions to large sets of algebraic equations gotten from the relations constraining associativity \cite[Definition 2.1.1]{tcat}.   It is well-known that any solution to this system of constraining equations corresponds to a categorification of the underlying Grothendieck ring (see \cite{BaKi}), therefore for each fusion category $\mathcal{C}$ and Galois automorphism $\sigma\in\mathrm{Gal}(\overline{\mathbb{Q}}/\mathbb{Q})$ one can construct a Galois conjugate fusion category $\mathcal{C}^\sigma$ by applying $\sigma$ to all structure constants of $\mathcal{C}$.  \emph{A priori} $\mathcal{C}^\sigma$ could be equivalent to $\mathcal{C}$, but many new examples of fusion categories can be created by this Galois conjugation. 

\subsubsection{Dimensions}\label{dimsec}

There are many methods for differentiating between fusion categories via numerical invariants.  For each $X\in\mathcal{O}(\mathcal{C})$ we define $\mathrm{FPdim}(X)$ to be the maximal real eigenvalue of the nonnegative integer matrix of $\otimes$-ing with $X$ afforded by the Frobenius-Perron Theorem \cite{frobenius,perron}. These dimensions collect to a ring homomorphism $\mathrm{FPdim}:K(\mathcal{C})\to\mathbb{C}$ \cite[Proposition 3.3.6 (1)]{tcat}.  Then the \emph{Frobenius-Perron dimension of }$\mathcal{C}$
\begin{equation}
\mathrm{FPdim}(\mathcal{C}):=\sum_{X\in\mathcal{O}(\mathcal{C})}\mathrm{FPdim}(X)^2
\end{equation}
is one measurement of the size of a fusion category.  Assumptions about Frobenius-Perron dimension are very limiting as this is a numerical invariant not only of the fusion category but of its underlying Grothendieck ring.  More flexible numerical invariants are the \emph{squared norms} $|X|^2$ for $X\in\mathcal{O}(\mathcal{C})$ introduced by M\"uger in \cite{mug2} where one can find further details (we will have no reason to refer to these in further sections).  As with Frobenius-Perron dimension, we then define the \emph{global dimension of }$\mathcal{C}$:
\begin{equation}\label{globaldimension}
\mathrm{dim}(\mathcal{C}):=\sum_{X\in\mathcal{O}(\mathcal{C})}|X|^2.
\end{equation}
As global dimension is defined in terms of morphisms of a fusion category $\mathcal{C}$, for all $\sigma\in\mathrm{Gal}(\overline{\mathbb{Q}}/\mathbb{Q})$ we have $\dim(\mathcal{C}^\sigma)=\sigma(\dim(\mathcal{C}))$ while $\mathrm{FPdim}(\mathcal{C}^\sigma)=\mathrm{FPdim}(\mathcal{C})$ because the underlying Grothendieck ring is unchanged.  This makes the category
\begin{equation}
\overline{\mathcal{C}}:=\boxtimes_{\sigma\in\mathrm{Gal}(\mathbb{Q}(\dim(\mathcal{C}))/\mathbb{Q})}\mathcal{C}^\sigma
\end{equation}
incredibly useful where $\boxtimes$ is the \emph{Deligne product} of fusion categories \cite[Section 4.6]{tcat}.  In particular, the global dimension is $\dim(\overline{\mathcal{C}})=N(\dim(\mathcal{C}))$ while the Frobenius-Perron dimension is $\mathrm{FPdim}(\overline{\mathcal{C}})=\mathrm{FPdim}(\mathcal{C})^{d_{\dim(\mathcal{C})}}$, recalling that $d_{\dim(\mathcal{C})}=[\mathbb{Q}(\dim(\mathcal{C})):\mathbb{Q}]$.

\subsubsection{Sphericality}

A \emph{pivotal structure} $\delta$ on a fusion category $\mathcal{C}$ is a collection of functorial tensor isomorphisms $\delta_X:X\to (X^\ast)^\ast$ for all $X\in\mathcal{O}(\mathcal{C})$.  These may be represented by invertible scalars $\dim_\delta(X)$, the \emph{categorical dimensions} of simple objects, which collect to form a ring homomorphism $\mathrm{dim}_\delta:K(\mathcal{C})\to\mathbb{C}$ \cite[Proposition 4.7.12]{tcat}.   We say a pivotal structure $\delta$ is \emph{spherical} if $\dim_\delta(X)=\dim_\delta(X^\ast)$ for all $X\in\mathcal{O}(\mathcal{C})$.  It is a slight (common) abuse of language to say $\delta$ is a \emph{spherical structure} rather than $\delta$ is a pivotal structure which is spherical.  A \emph{spherical fusion category} is a fusion category $\mathcal{C}$ with a chosen spherical structure $\delta$.  But in what follows, the choice of spherical structure is irrelevant or uniquely determined so we will omit $\delta$ from all notation.  If $\mathcal{C}$ is a spherical fusion category, one has $\dim(X)=|X|^2$ for $X\in\mathcal{O}(\mathcal{C})$ \cite[Corollary 2.10]{ENO} so one may replace squared norms with categorical dimensions in (\ref{globaldimension}).  Sphericality is potentially a very weak assumption on a fusion category as all currently known examples of fusion categories posess a spherical structure.  Whether this is true in general has been an open problem of great interest.  When $\dim(\mathcal{C})=\mathrm{FPdim}(\mathcal{C})$ we say that $\mathcal{C}$ is \emph{pseudounitary}.  Pseudounitary fusion categories have a unique spherical structure such that the categorical and Frobenius-Perron dimensions of all simple objects coincide \cite[Proposition 8.23]{ENO}.  It will be assumed that a pseudounitary fusion category is equipped with this distinguished spherical structure unless otherwise indicated.

\begin{example}
If $\mathcal{C}$ is a fusion category such that $\mathrm{FPdim}(X)=1$ for all $X\in\mathcal{O}(\mathcal{C})$ (such $X$ are called \emph{invertible}), we say that $\mathcal{C}$ is \emph{pointed}.  Pointed categories are pseudounitary \cite[Proposition 8.24]{ENO} with $\mathrm{FPdim}(\mathcal{C})=\dim(\mathcal{C})=|\mathcal{O}(\mathcal{C})|$ and moreover are equivalent to $\mathrm{Vec}_G^\omega$ for some finite group $G$ and $3$-cocycle $\omega$, as fusion categories.
\end{example}

\subsubsection{Modularity}

The correct notion of $\otimes$-commutativity in fusion categories is that of a braiding.  We say a fusion category $\mathcal{C}$ is \emph{braided} if there exists a family of functorial isomorphisms $\sigma_{X,Y}:X\otimes Y\to Y\otimes X$ for all $X,Y\in\mathcal{O}(\mathcal{C})$ satisfying braid-like relations \cite[Definition 8.1.1]{tcat} which will be irrelevant for our arguments.  The complexity of a braiding is measured by the \emph{symmetric center} $\mathcal{C}'$ which has simple objects $X$ such that $\sigma_{Y,X}\sigma_{X,Y}=\mathrm{id}_{X\otimes Y}$ for all $Y\in\mathcal{O}(\mathcal{C})$.  The two extremes are braided fusion categories for which $\mathcal{C}'=\mathcal{C}$ or $\mathcal{C}'=\mathrm{Vec}$ which we call \emph{symmetrically braided} and $\emph{nondegenerately braided}$, respectively.  A spherical fusion category which is nondegenerately braided is a \emph{modular tensor category}.  Modular tensor categories are so-named because they give two important $|\mathcal{O}(\mathcal{C})|$-dimensional representations of the modular group $SL_2(\mathbb{Z})$ generated by $x,y$ such that $x^2=1$ and $(xy)^3=1$.  First, there exists a projective representation $x\mapsto\tilde{S}$, $y\mapsto\tilde{T}$ where $\tilde{T}$ is a diagonal matrix consisting of roots of unity $\tilde{t}_X$ for $x\in\mathcal{O}(\mathcal{C})$ \cite[Corollary 8.18.2]{tcat} known as \emph{twists} of simple objects.  The collection of all scalars $\tilde{s}_{X,Y}$ and $\tilde{t}_X$ for all $X,Y\in\mathcal{O}(\mathcal{C})$ is called the \emph{unnormalized modular data} of $\mathcal{C}$.  We may then define the \emph{Gauss sums}
\begin{equation}
\tau^\pm:=\sum_{X\in\mathcal{O}(\mathcal{C})}(\tilde{t}_X)^{\pm1}\dim(X)^2.
\end{equation}
Choose any $\gamma\in\mathbb{C}$ which is a cube root of the root of unity $\tau^+/\sqrt{\dim(\mathcal{C})}$ where $\sqrt{\dim(\mathcal{C})}$ is the positive square root of the global dimension.  Then $s:=(1/\sqrt{\dim(\mathcal{C})})\tilde{s}$ and $t:=\gamma^{-1}\tilde{t}$ define a linear representation of $SL_2(\mathbb{Z})$.   The collection of all scalars $s_{X,Y}$ and $t_X$ for all $X,Y\in\mathcal{O}(\mathcal{C})$ is called the \emph{normalized modular data} of $\mathcal{C}$. 

\subsubsection{Galois action}

A powerful characteristic of a modular tensor category is the existence of a Galois action on its normalized modular data \cite{gannoncoste,de1991markov}.  Following \cite[Theorem II]{dong2015congruence}, let $n,\tilde{n}\in\mathbb{Z}_{\geq1}$ be the orders of $t$ and $\tilde{t}$, respectively.  Then the normalized and unnormalized modular data of $\mathcal{C}$ are contained in $\mathbb{Q}(\zeta_n)$ and $\mathbb{Q}(\zeta_{\tilde{n}})$, respectively, with $\tilde{n}\mid n$, i.e. \!$\mathbb{Q}(\zeta_{\tilde{n}})\subseteq\mathbb{Q}(\zeta_n)$.  And for each $\sigma\in\mathrm{Gal}(\mathbb{Q}(\zeta_n)/\mathbb{Q})$, there exists a unique permutation $\hat{\sigma}$ of $\mathcal{O}(\mathcal{C})$ such that
\begin{equation}\label{s}
\sigma\left(\dfrac{s_{X,Y}}{s_{\mathbbm{1},Y}}\right)=\dfrac{s_{X,\hat{\sigma}(Y)}}{s_{\mathbbm{1},\hat{\sigma}(Y)}}
\end{equation}
and
\begin{equation}
\sigma^2(t_X)=t_{\hat{\sigma}(X)}
\end{equation}
for all $X,Y\in\mathcal{O}(\mathcal{C})$.  Of particular interest is the square of Equation (\ref{s}) when $Y=\mathbbm{1}$ which describes the Galois action on squared categorical dimensions:
\begin{equation}
\dfrac{\sigma(\dim(X)^2)}{\sigma(\dim(\mathcal{C}))}=\dfrac{\dim(\hat{\sigma}(X)^2)}{\dim(\mathcal{C})}.
\end{equation}
Here we have used the fact that $\sigma(s_{X,Y}^2)=s_{X,\hat{\sigma}(Y)}^2=s_{\hat{\sigma}(X),Y}^2$.


\subsection{Formal codegrees}\label{formalcodegrees}

In \cite{codegrees}, Ostrik introduced the notion of the (multi-set of) formal codegrees of a fusion ring $R$, e.g. \!the Grothendieck ring of a fusion category (or more generally a \emph{based ring}).  Here we briefly recount the definition for completeness, although often in practice there are more efficient methods for computing and studying these scalars which we subsequently summarize.

\subsubsection{Definition and properties}

The complexified ring $S:=R\otimes_\mathbb{Z}\mathbb{C}$ where $R$ is a fusion ring possesses a nondegenerate pairing $S\times S^\ast\to\mathbb{C}$ and so for each irreducible representation $\varphi$ of $R$, we define $r_\varphi\in S$ as that which corresponds to $\mathrm{Tr}(\cdot,\varphi)\in S^\ast$ (trace in the sense of linear algebra).  The elements $r_\varphi\in S$ are central and thus act by a nonzero scalar $f_\varphi\in\mathbb{C}$ on $\varphi$ and by zero on any other irreducible representation of $R$.  The scalars $f_\varphi$ over all irreducible representations of $R$ (denoted $\mathrm{Irr}(R)$) are called the \emph{formal codegrees} of $R$ and more specifically, if $\mathcal{C}$ is a fusion category, we say the formal codegrees of $K(\mathcal{C})$ are the formal codegrees of $\mathcal{C}$ itself.  From this definition and \cite[Theorem 8.51]{ENO}, the formal codegrees of a fusion category are cyclotomic integers whose Galois conjugates are all greater than or equal to 1 \cite[Remark 2.12]{ost15}.  It was shown in \cite[Theorem 1.2]{codegrees} that the formal codegrees of a fusion category are algebraic $d$-numbers as well.

\par Let $B\subset R$ be a basis of a fusion ring with duality $b\mapsto b^\ast$ for all $b\in B$.  If $\varphi$ is the one-dimensional (hence, irreducible) representation induced by a ring homomorphism $g_\varphi:R\to\mathbb{C}$, then the corresponding formal codegree of $R$ is
\begin{equation}\label{ate}
f_\varphi=\sum_{b\in B}g_\varphi(b)g_\varphi(b^\ast)=\sum_{b\in B}|g_\varphi(b)|^2.
\end{equation}
In what follows we will not notationally differentiate between a ring homomorphism and the corresponding one-dimensional representation.  Let $\mathcal{C}$ be a fusion category with Grothendieck ring $K(\mathcal{C})$.  Then Frobenius-Perron theory provides one ring homomorphism: $\mathrm{FPdim}:K(\mathcal{C})\to\mathbb{C}$, and when $\mathcal{C}$ is spherical, categorical dimension provides another: $\dim:K(\mathcal{C})\to\mathbb{C}$ .  Moreover, if $\sigma\in\mathrm{Gal}(\overline{\mathbb{Q}}/\mathbb{Q})$, then one may conjugate any ring homomorphism $\varphi:R\to\mathbb{C}$ of a fusion ring $R$ by $\sigma$.  Specifically, $\sigma(\varphi):R\to\mathbb{C}$ is defined by $\sigma(\varphi)(x):=\sigma(\varphi(x))$ for all $x\in R$ and evidently $\sigma(\varphi)\cong\varphi$ if and only if $\sigma(\varphi(x))=\varphi(x)$ for all $x\in R$.

\par One should often think of formal codegrees of a fusion category $\mathcal{C}$ as an abstraction or generalization of global/Frobenius-Perron dimension for multiple reasons.  Firstly, $\mathrm{FPdim}(\mathcal{C})$ is the maximal real eigenvalue of the operator of $\otimes$-ing with $R:=\oplus_{X\in\mathcal{O}(\mathcal{C})}X\otimes X^\ast$.  But more generally \cite[Remark 2.11]{ost15}, the remaining spectrum of this operator consists of $\dim(\varphi)f_\varphi$ for all $\varphi\in\mathrm{Irr}(K(\mathcal{C}))$ appearing with multiplicity $\dim(\varphi)^2$.  In concrete examples this can be an efficient method for computing formal codegrees.  Secondly \cite[Corollary 2.14]{ost15}, each formal codegree $f$ divides $\dim(\mathcal{C})$ and by \cite[Proposition 2.10]{ost15} the following equalities hold:
\begin{equation}\label{twelve}
\sum_{\varphi\in\mathrm{Irr}(K(\mathcal{C}))}\dfrac{\dim(\varphi)}{f_\varphi}=1\qquad\iff\qquad\sum_{\varphi\in\mathrm{Irr}(K(\mathcal{C}))}\dim(\varphi)\dfrac{\dim(\mathcal{C})}{f_\varphi}=\dim(\mathcal{C}).
\end{equation}
For a modular tensor category, all irreducible representations of $K(\mathcal{C})$ are one-dimensional from commutativity, and Verlinde's formula implies the formal codegrees are $\dim(\mathcal{C})/\dim(X)^2$ over all $X\in\mathcal{O}(\mathcal{C})$, so (\ref{twelve}) is the definition of global dimension in (\ref{globaldimension}).


\subsubsection{A series of examples}

We compute the formal codegrees of several fusion categories to illustrate techniques and concepts used in further proofs.

\begin{example}\label{exampletoooo}
Let $n\in\mathbb{Z}_{\geq2}$.  A basis of $K(\mathrm{Vec}_{\mathbb{Z}/n\mathbb{Z}})$ is indexed by elements of, and has fusion rules mimicking, the group operation of $\mathbb{Z}/n\mathbb{Z}$.  As such, any irreducible representation is one-dimensional, corresponding to a ring homomorphism $\varphi:K(\mathrm{Vec}_{\mathbb{Z}/n\mathbb{Z}})\to\mathbb{C}$ determined by $\varphi(x)$ where $x\in\mathbb{Z}/n\mathbb{Z}$ is any generating element.  Further, $1=\varphi(1_R)=\varphi(x^n)=\varphi(x)^n$ and $\varphi(x)$ must be an $n$th root of unity.  Denote the ring homomorphisms $\varphi_j(x):=\zeta_n^j$ for $0\leq j\leq n-1$.  The corresponding one-dimensional representations $\varphi_0,\ldots,\varphi_{n-1}$ are all nonisomorphic, and many are Galois conjugate with one another (except $\varphi_0$) depending on $n$.  Using the formula in (\ref{ate}) we then have formal codegrees
\begin{equation}
\begin{array}{|c|ccccc|}
\hline \text{formal codegree}& n & n & \cdots & n & n \\\hline
\text{irr. rep.} & \dim=\mathrm{FPdim}=\varphi_0 & \varphi_1 & \cdots & \varphi_{n-2} & \varphi_{n-1} \\\hline
\end{array}
\end{equation} 
Observe that the formal codegrees are indistinguishable, but they correspond to distinct irreducible representations.  In particular, the irreducible representations corresponding to the global and Frobenius-Perron dimensions are identical, and no other irreducible representation is Galois conjugate to them.
\end{example}

\begin{example}\label{exampletwo}
Let $\mathrm{Fib}$ be the unique rank 2 spherical fusion category \cite{ostrik} with $\mathcal{O}(\mathrm{Fib})=\{\mathbbm{1},X\}$, nontrivial fusion rule $X\otimes X\cong\mathbbm{1}\oplus X$, and $\dim(X)=[2]\in\mathbb{Q}(\zeta_5)^+$, often called the Fibonacci category.  We will compute the formal codegrees of $\overline{\mathrm{Fib}}:=\mathrm{Fib}\boxtimes\mathrm{Fib}^\sigma$, where $\sigma$ is the nontrivial Galois automorphism of $\mathbb{Q}(\zeta_5)^+$.  One may easily compute $\mathrm{FPdim}(\overline{\mathrm{Fib}})=5[2]^2$ and $\dim(\overline{\mathrm{Fib}})=5$.  Therefore $f_{\mathrm{FPdim}}=5[2]^2$, $f_{\sigma(\mathrm{FPdim})}=5\sigma([2])^2$, and $f_{\dim}=5$ are formal codgrees of $\overline{\mathrm{Fib}}$.  But $\dim$ has a ``hidden'' formal codegree: although $\sigma$ fixes the formal codegree $5$, it does not fix the representation $\dim$ itself.  In particular the values $\dim$ takes on the nontrivial objects of $\mathrm{Fib}$ and $\mathrm{Fib}^\sigma$ are swapped by $\sigma$ by design.  Therefore the cardinality of the Galois orbit of $\dim$ is 2, while the cardinality of the Galois orbit of the corresponding formal codegrees is 1.  Both orbits have cardinality 2 for the representation $\mathrm{FPdim}$.
\begin{equation}
\begin{array}{|c|cccc|}
\hline\text{formal codegree} & 5 & 5 & 5[2]^2 & 5\sigma([2])^2 \\\hline
\text{irr. rep.} & \dim & \sigma(\dim) & \mathrm{FPdim} & \sigma(\mathrm{FPdim})  \\\hline
\end{array}
\end{equation}
This is a complete collection of formal codegrees of $\overline{\mathrm{Fib}}$ using (\ref{twelve}) since
\begin{equation}
\dfrac{1}{5}+\dfrac{1}{5}+\dfrac{1}{5[2]^2}+\dfrac{1}{5\sigma([2])^2}=1.
\end{equation}
\end{example}

\begin{example}
Let $\mathcal{C}$ be the spherical fusion category constructed via the representation theory of $\mathcal{U}_q(\mathfrak{so}_5)$ where $q^2$ is a ninth root of unity \cite[Section 5.2]{rowell}.  In particular, everything we consider below takes place in $\mathbb{Q}(\zeta_9)^+$.  This extension has degree 3 so let $\sigma$ be the generating Galois automorphism such that $\sigma(\cos(\pi/9))=\cos(2\pi/9)$.  The adjoint subcategory $\mathcal{C}_\mathrm{ad}$ \cite[Section 3.1]{nilgelaki} is rank $6$ and with a particular ordering of the basis, the Frobenius-Perron and categorical dimensions of simple objects are, respectively, $1,\sigma([2])[4],[2][4],[2],[2],[2]$ and $1,1,-1,[2],\sigma([2]),\sigma^2([2])$.  Hence all Galois conjugates (by $\sigma$, $\sigma^2$) of $\dim$ and $\mathrm{FPdim}$ are nonisomorphic and using (\ref{ate}) one should verify the following collection of formal codegrees of $\mathcal{C}_\mathrm{ad}$ is complete using (\ref{twelve}).
\begin{equation}
\begin{array}{|c|cccccc|}
\hline\text{formal codegree}&9 & 9 & 9 & 9[4]^2 & 9\sigma([4])^2 & 9\sigma^2([4])^2 \\\hline
\text{irr. rep.}& \dim & \sigma(\dim) & \sigma^2(\dim) & \mathrm{FPdim} & \sigma(\mathrm{FPdim}) & \sigma^2(\mathrm{FPdim})  \\\hline
\end{array}
\end{equation}
In this example we see that rational integer global dimensions arise outside of weakly integral examples and ``constructed'' examples by taking Deligne products over all Galois conjugates.
\end{example}

\begin{example}\label{neargroup}
As a final example we will compute the formal codegrees of the \emph{near-group} fusion category $\mathcal{N}(G,k)$ \cite{MR3167494,izumi2001structure} where $G$ is a finite abelian group and $k\in\mathbb{Z}_{\geq1}$ using \cite[Lemma 2.6]{codegrees}.  We have $|\mathcal{O}(\mathcal{N}(G,k))|=|G|+1$: $|G|$ of the simple objects are invertible and have the fusion rules of the group $G$ and for each $g\in G$, the remaining simple object $\rho$ satisfies $\rho g=g\rho=\rho$.  The only non-trivial fusion rule is $\rho^2=k|G|\rho\oplus\bigoplus_{g\in G}g$.  Thus $\mathrm{FPdim}(\rho)$ is a root of $x^2-k|G|x-|G|$.  The matrix of $\otimes$-ing with $R=\rho^2\oplus|G|\mathbbm{1}$ is
\begin{equation}\label{matrix}
\left[\begin{array}{ccccc}
|G|+1 & 1 & \cdots & 1 & k|G| \\
1 & |G|+1 & \cdots & 1 & k|G| \\
\vdots & \vdots & \ddots & \vdots & \vdots \\
1 & \cdots & \cdots & |G|+1 & k|G| \\
k|G| & k|G| & \cdots & k|G| & |G|(k^2|G|+2)
\end{array}\right].
\end{equation}
As an exercise one should compute that the eigenvalues of $(\ref{matrix})$ are $|G|$ with multiplicity $|G|-1$ and the roots of $x^2-|G|(k^2|G|+4)x+|G|^2(k^2|G|+4)$ which are $\mathrm{FPdim}(\mathcal{N}(G,k))$ and its Galois conjugate.  The fusion rules of $\mathcal{N}(G,k)$ are commutative, thus all irreducible representations of its Grothendieck ring are one-dimensional.  Moreover these eigenvalues are the formal codegrees of $\mathcal{N}(G,k)$ on the nose.
\end{example}


\section{Norm finiteness}\label{nooorm}

For any $\alpha\in\mathbb{A}$ whose Galois conjugates are all greater than or equal to $1$, it is clear that $\alpha\leq N(\alpha)$.  Hence for any fusion category $\mathcal{C}$, $\mathrm{FPdim}(\mathcal{C})\leq N(\mathrm{FPdim}(\mathcal{C}))$ and moreover there are finitely-many fusion categories $\mathcal{C}$ for any fixed $N(\mathrm{FPdim}(\mathcal{C}))\in\mathbb{Z}_{\geq1}$ by \cite[Corollary 3.13]{paul}.  The similar result for global dimension is more subtle as it is not known if the set of all global dimensions of fusion categories (or even spherical fusion categories) is discrete \cite[Section 1.2]{ostrikremarks}.  Recall the definition $\overline{\mathcal{C}}:=\boxtimes_{\sigma\in\mathrm{Gal}(\mathbb{Q}(\dim(\mathcal{C}))/\mathbb{Q})}\mathcal{C}^\sigma$ from Section \ref{dimsec} for the following proof.

\begin{theorem}\label{normfiniteness}
Let $m\in\mathbb{Z}_{\geq1}$.  There exist finitely-many fusion categories $\mathcal{C}$ up to equivalence with $N(\dim(\mathcal{C}))=m$.
\end{theorem}

\begin{proof}
Set $g:=\dim(\mathcal{C})$ for brevity.  It is clear that $m=N(g)=1$ if and only if $\mathcal{C}\simeq\mathrm{Vec}$.  Ostrik \cite[Theorem 4.1.1(i)]{ostrikremarks} has shown that if $\mathcal{C}$ is a nontrivial spherical fusion category, then $g>4/3$.  This implies that if $m=N(g)>1$, then $d_g\leq\left\lfloor\log(m)/\log(4/3)\right\rfloor:=k$.  Theorem 5.1.1 of \cite{ostrikremarks} states
\begin{equation}
\Sigma:=\{\text{eq. \!classes of }\overline{\mathcal{D}}:\mathcal{D}\text{ fusion category and }\dim(\overline{\mathcal{D}})=m\}
\end{equation}
is finite.  Thus $\{\mathrm{FPdim}(\overline{\mathcal{D}}):\overline{\mathcal{D}}\in \Sigma\}$ is finite.  But for each degree extension $1\leq d_g\leq k$, we have $\mathrm{FPdim}(\mathcal{C})=(\mathrm{FPdim}(\overline{\mathcal{C}}))^{1/d_g}$ which implies
\begin{equation}
\{\mathrm{FPdim}(\mathcal{D}):\mathcal{D}\text{ spherical fusion category and }\overline{\mathcal{D}}\in \Sigma\}
\end{equation}
is finite, and the spherical case is complete as the number of fusion categories of fixed Frobenius-Perron dimension (up to equivalence) is finite.  This implies norm finiteness for arbitrary fusion categories because each fusion category $\mathcal{C}$ posesses a \emph{sphericalization} $\tilde{\mathcal{C}}$ \cite[Section 5.3]{MR3523185} with $\dim(\tilde{\mathcal{C}})=2g$ (hence $\mathbb{Q}(\dim(\mathcal{C}))=\mathbb{Q}(\dim(\tilde{\mathcal{C}}))$), and moreover $N(\dim(\tilde{\mathcal{C}}))=2^{d_g}\cdot m$ where, as noted above, $d_g\leq k$.  
\end{proof}

\begin{proposition}\label{ranknorm}
Let $\mathcal{C}$ be a fusion category and $g:=\dim(\mathcal{C})$.  Then
\begin{equation}\label{doce}
d_g^{d_g}\leq|\mathcal{O}(\mathcal{C})|^{d_g}\leq N(g).
\end{equation}
\end{proposition}

\begin{proof}
We have by \cite[Lemma 4.2.2]{ostrikremarks},
\begin{equation}
|\mathcal{O}(\mathcal{C})|^{d_g}=|\mathcal{O}(\overline{\mathcal{C}})|\leq\dim(\overline{\mathcal{C}})=N(g).
\end{equation}
The left-hand inequality in (\ref{doce}) is the fact that $d_g\leq|\mathcal{O}(\mathcal{C})|$ as the number of formal codegrees of $\mathcal{C}$ is less than or equal to the rank.
\end{proof}

If $\alpha\in\mathbb{A}$ is totally positive with minimal polynomial $x^n-a_1x^{n-1}+\cdots+(-1)^{n-1}a_{n-1}x+(-1)^na_n$, set $c_j(\alpha):=a_j$, a positive integer.

\begin{proposition}
Let $j,m\in\mathbb{Z}_{\geq1}$.  There exist finitely-many fusion categories $\mathcal{C}$ up to equivalence with $c_j(\dim(\mathcal{C}))=m$.
\end{proposition}

\begin{proof}
Note that if $\mathcal{C}$ is nontrivial, each of the Galois conjugates of $\dim(\mathcal{C})$ are greater than $4/3$ \cite[Theorem 4.1.1(i)]{ostrikremarks}.  Recalling that $m=c_j(\dim(\mathcal{C}))$ can be expressed as a symmetric function of its Galois conjugates,
\begin{equation}
m>\binom{d_{\dim(\mathcal{C})}}{j}\left(4/3\right)^j>(3/4)d_{\dim(\mathcal{C})},
\end{equation}
where the second inequality follows from $(4/3)^j>(3/4)j$ and $\binom{d_{\dim(\mathcal{C})}}{j}\geq d_{\dim(\mathcal{C})}/j$ for all $j\in\mathbb{Z}_{\geq1}$ with $j\leq d_{\dim(\mathcal{C})}$. Hence $k:=(4/3)m\geq d_{\dim(\mathcal{C})}$.  But by the $d$-number condition in Equation (\ref{ddd}), $N(\dim(\mathcal{C}))^j$ divides $m^{d_{\dim(\mathcal{C})}}$ for all $1\leq j\leq d_{\dim(\mathcal{C})}$.  In particular $N(\dim(\mathcal{C}))\leq\sqrt[j]{N^k}$, a fixed value.  Our result follows by Theorem \ref{normfiniteness}.
\end{proof}


\section{Rational integer divisors of formal codegrees}\label{divisors}

There has been much effort to study fusion categories whose Frobenius-Perron dimenson is a rational integer ($\mathcal{C}$ is \emph{weakly integral}) with small prime factorizations.  Some of the first examples of this are the classification of fusion categories whose Frobenius-Perron dimension is a rational prime \cite[Corollary 8.30]{ENO} or the square of a rational prime \cite[Proposition 8.32]{ENO}.  When $\mathcal{C}$ is not weakly integral one can study the more extreme case when $\mathrm{FPdim}(\mathcal{C})$ (hence the remainder of the formal codegrees of $\mathcal{C}$ by Theorem \ref{newprop}) has \emph{no} nontrivial rational integer divisors.  The most obvious case of this is if $\alpha$, a formal codegree of a fusion category, were a unit.  But any $\alpha\in\mathbb{A}^\times$ which is totally real is either 1, or $\sigma(\alpha)<1$ for some $\sigma\in\mathrm{Gal}(\overline{\mathbb{Q}}/\mathbb{Q})$ because $N(\alpha)=1$, hence $\alpha$ is not the formal codegree of any fusion category by \cite[Remark 2.12]{ost15}.  Thus the only fusion category with a unit for a formal codegree is $\mathrm{Vec}$ by \cite[Proposition 2.10]{ost15}.  Surprisingly, the rational integer divisors of a formal codegree of a fusion category can be identified by only measuring its norm.  This is far from true for arbitrary algebraic integers.

\begin{example}
Consider either root $\alpha$ of $x^2+ax+b$ where $a,b\in\mathbb{Z}\setminus\{0\}$, $\gcd(a,b)=1$, and $d_\alpha=2$.  Note that $\alpha$ is an algebraic integer, but $\alpha$ is not a $d$-number unless $b=\pm1$.  We have $N(\alpha)=b$ which may have any nonzero rational integer coprime to $a$ as a divisor.  But $\alpha$ itself is not divisible by any $m\in\mathbb{Z}\setminus\{0,\pm1\}$.  Indeed, $x^2+(a/m)x+(b/m^2)$ is the minimal polynomial of $\alpha/m$ over $\mathbb{Q}$ which lies in $\mathbb{Z}[x]$ if and only if $m\mid a$ and $m^2\mid b$ which cannot happen due to the coprime assumption.
\end{example}

\begin{proposition}\label{newlemma}
Let $\alpha,\beta\in\mathbb{A}$ be cyclotomic $d$-numbers (thus, nonzero) and $\mathbb{L}:=\mathbb{Q}(\alpha,\beta)$.  Then $\beta$ divides $\alpha$ if and only if $N_\mathbb{L}(\beta)$ divides $N_\mathbb{L}(\alpha)$.
\end{proposition}

\begin{proof}
By the $d$-number condition in Equation (\ref{ddd}), there exist $u_\alpha,u_\beta\in\mathbb{A}^\times$ such that
\begin{equation}
\alpha^{[\mathbb{L}:\mathbb{Q}]}=(\alpha^{d_\alpha})^{[\mathbb{L}:\mathbb{Q}(\alpha)]}=(N(\alpha)\cdot u_\alpha)^{[\mathbb{L}:\mathbb{Q}(\alpha)]}=N_\mathbb{L}(\alpha)\cdot u_\alpha^{[\mathbb{L}:\mathbb{Q}(\alpha)]}
\end{equation}
and similary for $\beta$.  Thus
\begin{equation}
\left(\alpha/\beta\right)^{[\mathbb{L}:\mathbb{Q}]}=(N_\mathbb{L}(\alpha)/N_\mathbb{L}(\beta))u_\alpha^{[\mathbb{L}:\mathbb{Q}(\alpha)]}u_\beta^{-[\mathbb{L}:\mathbb{Q}(\beta)]}.
\end{equation} 
Hence $\alpha/\beta\in\mathbb{A}$ if and only if $(\alpha/\beta)^{[\mathbb{L}:\mathbb{Q}]}\in\mathbb{A}$ if and only if $N_\mathbb{L}(\alpha)/N_\mathbb{L}(\beta)\in\mathbb{A}$.
\end{proof}

Proposition \ref{newlemma} is true for arbitrary algebraic $d$-numbers but requires subtleties which are not needed in what follows.  Let $J$ be an indexing set for the prime rational integers.

\begin{corollary}\label{noocore}
Let $\alpha\in\mathbb{A}$ be a cyclotomic $d$-number with $N(\alpha)=\pm\prod_{j\in J}p_j^{k_j}$.  The following are equivalent:
\begin{itemize}
\item[(a)] $0\leq k_j<d_\alpha$ for all $j\in J$, and
\item[(b)] for all $m\in\mathbb{Z}_{\geq2}$, $m$ does not divide $\alpha$.
\end{itemize}
\end{corollary}


\subsection{Results on general fusion categories}

\begin{theorem}\label{newprop}
Let $\mathcal{C}$ be a fusion category.  If $\mathcal{C}$ has a formal codegree $f$ such that for all $m\in\mathbb{Z}_{\geq2}$, $m$ does not divide $f$, then
\begin{itemize}
\item[(a)] all formal codegrees of $\mathcal{C}$ are Galois conjugate to $f$,
\item[(b)] $|\mathcal{O}(\mathcal{C})|=d_f$, and
\item[(c)] the Grothendieck ring of $\mathcal{C}$ is commutative.
\end{itemize}
\end{theorem}

\begin{proof}
Assume $f$ has minimal polynomial
\begin{equation}
x^n-a_1x^{n-1}+\cdots+(-1)^{n-1}a_{n-1}x+(-1)^na_n,
\end{equation}
where $n,a_j\in\mathbb{Z}_{\geq1}$ for all $1\leq j\leq n$ as $f$ is totally positive \cite[Remark 2.12]{ost15}.  If $f\in\mathbb{Z}$, then $f=1$ and $\mathcal{C}\simeq\mathrm{Vec}$ and the result is trivial, so we may assume $n>1$.  We have $a_n=N(f)$ and as $f$ is a $d$-number, $a_n^{n-1}$ divides $a_{n-1}^n$ \cite[Lemma 2.7(v)]{codegrees}.  Let $p_j$ be any rational prime dividing $a_n$.  If $p_j^{k_j}$ is the largest power of $p_j$ dividing $a_{n-1}$ and $p_j^{\ell_j}$ is the largest power of $p_j$ dividing $a_n$, then the largest power of $p_j$ dividing $a_{n-1}^n$ is $p_j^{nk_j}$.  Thus the above $d$-number condition implies $nk_j\geq(n-1)\ell_j$, or $k_j\geq \ell_j-\ell_j/n$.  Corollary \ref{noocore} implies that $\ell_j/n<1$, so $k_j\geq \ell_j$.  Hence $p_j^{\ell_j}$ divides $a_{n-1}$ for all primes $p_j$ dividing $a_n$, and therefore $a_{n-1}\geq a_n$.  This implies
\begin{equation}
\sum_{\sigma\in\mathrm{Gal}(\mathbb{Q}(f)/\mathbb{Q})}\dfrac{1}{\sigma(f)}=\dfrac{a_{n-1}}{a_n}\geq1.\label{theeq}
\end{equation}
Moreover this sum is precisely $1$ by \cite[Proposition 2.10]{ost15} and (a) follows.  Claim (b) follows as the number of distinct formal codegrees must be less than or equal to $|\mathcal{O}(\mathcal{C})|$, and claim (c) follows from \cite[Example 2.18]{ost15}.
\end{proof}

\begin{example}\label{sl2}
Consider the fusion categories $\mathcal{T}_\kappa:=\mathcal{C}(\mathfrak{sl}_2,\kappa-2)_\mathrm{ad}$ for prime $\kappa\in\mathbb{Z}_{\geq3}$ (refer to \cite[Section 2.3]{schopieray2} for a basic introduction).  It is well-known that $\dim(\mathcal{T}_\kappa)=\kappa/(4\sin^2(\pi/\kappa))$.  Hence $\mathbb{Q}(\dim(\mathcal{C}))=\mathbb{Q}(\sin^2(\pi/\kappa))$.  Let $\varphi$ be the Euler totient function.
For any $n\in\mathbb{Z}_{\geq1}$,
\begin{equation}
[\mathbb{Q}(\sin(2\pi/n)):\mathbb{Q}]=\left\{\begin{array}{ccc}(1/2)\varphi(n) & : & n\equiv0\pmod{8} \\[0.25cm] (1/4)\varphi(n) & : & n\equiv4\pmod{8} \\[0.25cm] \varphi(n) & : & \mathrm{else}\end{array}\right.
\end{equation}
and $[\mathbb{Q}(\sin(\pi/\kappa)):\mathbb{Q}(\sin^2(\pi/\kappa))]=2$ when $\kappa$ is odd.  Thus
\begin{equation}
[\mathbb{Q}(\dim(\mathcal{T}_\kappa)):\mathbb{Q}]=(1/2)\varphi(\kappa)=(1/2)(\kappa-1).
\end{equation}
We leave it as an exercise to verify that $N(\dim(\mathcal{T}_\kappa))=\kappa^{(\kappa-3)/2}$.  In particular, Theorem \ref{newprop} implies the set of all formal codegrees of $\mathcal{T}_\kappa$ is the same as the set of Galois conjugates of $\dim(\mathcal{T}_\kappa)$.  Moreover, if $\kappa_1,\ldots,\kappa_n$ is any finite set of distinct odd primes and $\sigma_1,\ldots,\sigma_n$ are any (not necessarily distinct) elements of $\mathrm{Gal}(\overline{\mathbb{Q}}/\mathbb{Q})$, $\boxtimes_{j=1}^n\mathcal{T}_{\kappa_j}^{\sigma_j}$ satisfies the hypotheses of Theorem \ref{newprop} as well.
\end{example}

\begin{corollary}\label{cortooo}
Let $\mathcal{C}$ be a fusion category.  If the Grothendieck ring of $\mathcal{C}$ is noncommutative, for every formal codegree $f$ of $\mathcal{C}$ there exists an integer $n_f\in\mathbb{Z}_{\geq2}$ dividing $f$.
\end{corollary}

\begin{proof}
This is the contrapositive of Theorem \ref{newprop} (c).
\end{proof}

\begin{example}
By some measures, the fusion category $\mathcal{H}$ corresponding to the extended Haagerup subfactor \cite{MR2979509} is the most interesting fusion category with respect to global dimension that has been constructed at this time.  We know $|\mathcal{O}(\mathcal{H})|=8$ and it has \emph{noncommutative} fusion rules.  In particular, Corollary \ref{cortooo} implies the norm of every formal codegree of $\mathcal{H}$ must contain a prime power factor $p_j^{k_j}$ with $k_j\geq3$.  For example, $\dim(\mathcal{H})=\mathrm{FPdim}(\mathcal{H})$ is the largest root of the polynomial
\begin{equation}
x^3-585x^2+8450x-21125.
\end{equation}
Moreover $N(\dim(\mathcal{H}))=21125=5^3\cdot13^2$ as predicted.
\end{example}

Recall that if $\mathbb{K}$ is any number field, each fusion category posesses a fusion subcategory $\mathcal{C}_\mathbb{K}$ whose objects are all $X\in\mathcal{C}$ such that $\mathrm{FPdim}(X)\in\mathbb{K}$ \cite[Proposition 1.6]{gannonschopieray}.  Thus $\mathcal{C}_\mathbb{Q}$ is the largest integral fusion subcategory of $\mathcal{C}$.  Alternatively, one can consider the fusion subcategory $\mathcal{C}_\mathrm{ad}$, the trivial component of the universal grading of $\mathcal{C}$ \cite[Section 3.2]{nilgelaki}.

\begin{corollary}\label{corhalf}
Let $\mathcal{C}$ be a fusion category. If $\mathcal{C}$ has a formal codegree satisfying the equivalent conditions of  Corollary \ref{noocore}, then $\mathcal{C}_\mathbb{Q}=\mathrm{Vec}$ and $\mathcal{C}=\mathcal{C}_\mathrm{ad}$.
\end{corollary}

\begin{proof}
As $\mathrm{FPdim}(\mathcal{C})$ is a formal codegree, it cannot be divisible by any integer $m\in\mathbb{Z}_{\geq2}$ by assumption.  However, $\mathrm{FPdim}(\mathcal{C}_\mathbb{Q})\in\mathbb{Z}$ and divides $\mathrm{FPdim}(\mathcal{C})$ \cite[Proposition 8.15]{ENO}, hence it must be trivial.  Similarly, if $U(\mathcal{C})$ is the universal grading group of $\mathcal{C}$, $\mathrm{FPdim}(\mathcal{C})=|U(\mathcal{C})|\mathrm{FPdim}(\mathcal{C}_\mathrm{ad})$ \cite[Theorem 3.5.2]{tcat}.  Thus $|U(\mathcal{C})|=1$ and $\mathcal{C}=\mathcal{C}_\mathrm{ad}$.
\end{proof}

\begin{lemma}\label{thelemma}
Let $\mathcal{C}$ be a fusion category.  If $\sigma\in\mathrm{Gal}(\mathbb{Q}(\dim(\mathcal{C}))/\mathbb{Q})$ is nontrivial, then $\mathcal{C}\boxtimes\mathcal{C}^\sigma$ is not Galois conjugate to a pseudounitary fusion category.
\end{lemma}

\begin{proof}
We have $\mathrm{FPdim}(\mathcal{C}\boxtimes\mathcal{C}^\sigma)=\mathrm{FPdim}(\mathcal{C})^2$.  Let $s_1,\ldots,s_k$ be the distinct Galois conjugates of $\dim(\mathcal{C})$. If $\tau\in\mathrm{Gal}(\overline{\mathbb{Q}}/\mathbb{Q})$ is any Galois automorphism, then for some indices $1\leq i,j\leq k$,
\begin{equation}
\tau(\dim(\mathcal{C}\boxtimes\mathcal{C}^\sigma))=\tau(\dim(\mathcal{C}))\tau(\sigma(\dim(\mathcal{C})))=s_is_j.
\end{equation}
As $\dim(\mathcal{C})\neq\sigma(\dim(\mathcal{C}))$ by assumption, $s_i\neq s_j$.  By \cite[Proposition 8.22]{ENO}, we have $s_i\leq\mathrm{FPdim}(\mathcal{C})$ for all $1\leq i\leq k$. Thus $\tau(\dim(\mathcal{C}\boxtimes\mathcal{C}^\sigma))=s_is_j<\mathrm{FPdim}(\mathcal{C}\boxtimes\mathcal{C}^\sigma)$ which is to say $\mathcal{C}\boxtimes\mathcal{C}^\sigma$ is not Galois conjugate to a pseudounitary fusion category.
\end{proof}


\subsection{Results on spherical fusion categories}

\begin{lemma}\label{cor1}
Let $\mathcal{C}$ be a spherical fusion category. If $\mathcal{C}$ has a formal codegree satisfying either of the equivalent conditions of Corollary \ref{noocore}, then $\mathcal{C}$ is Galois conjugate to a pseudounitary fusion category.
\end{lemma}

\begin{proof}
When $\mathcal{C}$ is spherical, $\dim(\mathcal{C})$ is a formal codegree of $\mathcal{C}$ and Theorem \ref{newprop} implies it lies in the Galois orbit of $\mathrm{FPdim}(\mathcal{C})$.
\end{proof}

\begin{proposition}\label{bigone}
Let $\mathcal{C}$ be a spherical fusion category, and $f$ be a formal codegree of $\mathcal{C}$ whose norm has prime factorization $N(f)=\prod_{j\in J}p_j^{k_j}$.  There exists $j\in J$ such that $k_j\geq\frac{1}{2}d_f$.
\end{proposition}

\begin{proof}
If not, $k_j<\frac{1}{2}d_f$ for all $j\in J$ and thus all formal codegrees of $\mathcal{C}$ are Galois conjugate, and have the same norm by Theorem \ref{newprop}.  Therefore, for any nontrivial $\sigma\in\mathrm{Gal}(\mathbb{Q}(\dim(\mathcal{C}))/\mathbb{Q})$, $\mathrm{FPdim}(\mathcal{C})^2$ is a formal codegree of $\mathcal{C}\boxtimes\mathcal{C}^\sigma$ which satisfies the equivalent conditions of Corollary \ref{noocore}.  Moreover Lemma \ref{cor1} implies $\mathcal{C}\boxtimes\mathcal{C}^\sigma$ is pseudounitary, contradicting Lemma \ref{thelemma}.
\end{proof}

\begin{corollary}\label{two}
Let $\mathcal{C}$ be a spherical fusion category.  If $\mathcal{C}$ has a formal codegree $f$ with square-free norm, then $d_f=[\mathbb{Q}(f):\mathbb{Q}]$ is $1$ or $2$.
\end{corollary}

\begin{theorem}\label{prop6}
Let $\mathcal{C}$ be a spherical fusion category with a formal codegree $f$ with square-free norm.  If $f\not\in\mathbb{Z}$, then $f=(1/2)(5\pm\sqrt{5})$ and $\mathcal{C}$ is equivalent to $\mathrm{Fib}$ or $\mathrm{Fib}^\sigma$ as a spherical fusion category.
\end{theorem}

\begin{proof}
Corollary \ref{two} implies $d_f=2$, hence all formal codegrees are Galois conjugate, i.e. \!$f$ is a Galois conjugate of $\dim(\mathcal{C})$.  Let $N:=N(f)$.   Then $f$ is a root of $x^2-Nx+N$ (see Equation (\ref{theeq})), hence $f=(1/2)(N\pm\sqrt{N^2-4N})$ which both must be greater than $4/3$ by \cite[Theorem 4.1.1]{ostrikremarks}.  This condition on the lesser of the two implies
\begin{align}
&&(1/2)(N-\sqrt{N^2-4N})&>4/3 \\
\Rightarrow&&(N-8/3)^2&>N^2-4N \\
\Rightarrow&&16/3&>N.
\end{align}
Thus $N=5$ because $N^2-4N<0$ when $N=1,2,3$, and $f\in\mathbb{Z}$ when $N=4$.  All spherical fusion categories of $\dim(\mathcal{C})=(1/2)(5\pm\sqrt{5})$ were classified in \cite[Example 5.1.2(iv)]{ostrikremarks} and shown to be equivalent to $\mathrm{Fib}$ or $\mathrm{Fib}^\sigma$.
\end{proof}


\section{Norm of global dimension is prime}\label{globp}

Here we prove that aside from $\mathrm{Fib}$, $\mathrm{Fib}^\sigma$, and $\overline{\mathrm{Fib}}$, all spherical braided fusion categories whose global dimension has prime norm $p\in\mathbb{Z}_{\geq2}$ are pointed.  By Theorem \ref{prop6} this reduces to a classification of spherical braided fusion categories global dimension exactly $p$.

\subsection{Results on spherical fusion categories}

\begin{lemma}\label{plemma}
Let $\mathcal{C}$ be a spherical fusion category.  If $\dim(\mathcal{C})=p\in\mathbb{Z}_{\geq2}$ is prime, then for all formal codegrees $f$ of $\mathcal{C}$, $f=p\cdot u_f$ for some $u_f\in\mathcal{O}_{\mathbb{Q}(f)}^\times$.
\end{lemma}

\begin{proof}
We know $f$ divides $p$ \cite[Corollary 2.14]{ost15}, thus $N(f)=p^k$ for some $0\leq k\leq d_f$.  Clearly $k\neq0$ as $\mathcal{C}$ must be nontrivial.  But if $0<k<d_f$, then Lemma \ref{cor1} implies $\mathcal{C}$ is (Galois conjugate to) a pseudounitary fusion category.  Fusion categories of Frobenius-Perron dimension $p$ are pointed \cite[Corollary 8.30]{ENO}, hence our claim is proven with $u_f=1$.  Otherwise $N(f)=p^{d_f}$, and moreover $f^{d_f}=p^{d_f}\cdot u$ for some $u\in\mathcal{O}_{\mathbb{Q}(f)}^\times$ by Equation (\ref{ddd}) and our claim follows.
\end{proof}

\begin{proposition}\label{dfdivides}
Let $\mathcal{C}$ be a spherical fusion category of prime global dimension $p\neq2$.  For all formal codegrees $f$ of $\mathcal{C}$, $d_f$ divides $(p-1)/2$.
\end{proposition}

\begin{proof}
The modular data of $\mathcal{Z}(\mathcal{C})$ is defined over $\mathbb{Q}(\zeta_{p^n})$ for some $n\in\mathbb{Z}_{\geq1}$ by \cite[Theorem 3.9]{paul} as $\dim(\mathcal{Z}(\mathcal{C}))=p^2$ where $\mathcal{Z}(\mathcal{C})$ is the \emph{Drinfeld center} of $\mathcal{C}$ \cite[Section 7.13]{tcat}.  The field generated by the modular data of $\mathcal{Z}(\mathcal{C})$ includes the formal codegrees of $\mathcal{C}$ by \cite[Theorem 2.13]{ost15}.  But $[\mathbb{Q}(\zeta_{p^n}):\mathbb{Q}]=p^{n-1}(p-1)$.  In particular, the formal codegrees of $\mathcal{C}$ are totally real \cite[Remark 2.12]{ost15}.  Thus $\mathbb{Q}(f)\subset\mathbb{Q}(\zeta_{p^n})^+$,  and $d_f$ divides $p^{n-1}(p-1)/2$.  But $\mathcal{C}$ has at most $|\mathcal{O}(\mathcal{C})|\leq p$ formal codegrees with equality if and only if $\mathcal{C}$ is pointed \cite[Lemma 4.2.2]{ostrikremarks}, including $\dim(\mathcal{C})=p$, hence $d_f<p-1$ and our claim follows.
\end{proof}

\begin{corollary}
Let $\mathcal{C}$ be a spherical fusion category of prime global dimension $p\in\mathbb{Z}_{\geq2}$.  Then $f\in\mathbb{Q}(\zeta_p)^+$ for all formal codegrees $f$ of $\mathcal{C}$.
\end{corollary}

\begin{proof}
As $\mathbb{Q}(\zeta_{p^n})/\mathbb{Q}$ is a cyclic extension, intermediate subfields are in one-to-one correspondence with rational integer divisors of $p^{n-1}(p-1)$.  E.g. \!$\mathbb{Q}(\zeta_p)\subset\mathbb{Q}(\zeta_{p^n})$ is the unique subfield $\mathbb{K}\subset\mathbb{Q}(\zeta_{p^n})$ with $[\mathbb{K}:\mathbb{Q}]=p-1$.  Moreover $[\mathbb{Q}(f):\mathbb{Q}]$ divides $p-1$ thus $\mathbb{Q}(f)\subset\mathbb{Q}(\zeta_p)^+\subset\mathbb{Q}(\zeta_p)$ as $f$ is totally real.
\end{proof}

\begin{proposition}\label{notlapland}
Let $\mathcal{C}$ be a spherical fusion category with prime global dimension $p\in\mathbb{Z}_{\geq2}$.  Then
\begin{equation}\label{equation20}
\sum_{\varphi\in\mathrm{Irr}(K(\mathcal{C}))}\dim(\varphi)u_{f_{\bm{\varphi}}}^{-1}=p,
\end{equation}
where $u_{f_\varphi}\in\mathcal{O}_{\mathbb{Q}(\zeta_p)^+}^\times$ are the units from Lemma \ref{plemma}.
\end{proposition}

\begin{proof}
This follows from \cite[Proposition 2.10]{ost15} and Lemma \ref{plemma}.
\end{proof}

\begin{theorem}\label{previous}
Let $p=2q+1$ where $p,q\in\mathbb{Z}_{\geq2}$ are both prime.  If $\mathcal{C}$ is a spherical fusion category with $N(\dim(\mathcal{C}))=p$, then $\mathcal{C}$ is pointed or $p=5$.
\end{theorem}

\begin{proof}
Proposition \ref{dfdivides} implies $d_f=1,q$ for all formal codegrees $f$.  This implies that either $\mathrm{FPdim}(\mathcal{C})=p$ and $\mathcal{C}$ is pointed, or $d_{\mathrm{FPdim}(\mathcal{C})}=q$.  
But \cite[Corollary 2.15]{ost15} states that all formal codegrees of $\mathcal{C}$ are contained in the field of categorical dimensions:
\begin{equation}
\mathbb{Q}(\dim(X):X\in\mathcal{O}(\mathcal{C})).
\end{equation}
In particular the categorical dimension homomorphism $\dim:K(\mathcal{C})\to\mathbb{C}$ has at least $q$ nonisomorphic Galois conjugates.  Therefore there are precisely $2q$ formal codegrees of $\mathcal{C}$: $p$ with multiplicity $q$ and the Galois conjugates of $\mathrm{FPdim}(\mathcal{C})$.  But Siegel's trace bound for totally positive algebraic integers \cite[Theorem III]{siegel} implies $\mathrm{Tr}(u_{\mathrm{FPdim}(\mathcal{C})}^{-1})\geq3q/2\geq q+1$ with equality if and only if $q=2$.  Hence Equation (\ref{equation20}) from Proposition \ref{notlapland} can only hold when $p=5$ or $\mathcal{C}$ is pointed.
\end{proof}

\begin{note}
The primes $p\in\mathbb{Z}_{\geq2}$ of the form $p=2q+1$ where $q\in\mathbb{Z}_{\geq2}$ is prime as well are known as \emph{safe primes}.  The safe primes less than $500$ are
\begin{equation}
5, 7, 11, 23, 47, 59, 83, 107, 167, 179, 227, 263, 347, 359, 383, 467,\text{ and }479.
\end{equation}
\end{note}

\begin{example}\label{thurteen}
Theorem \ref{previous} and \cite[Example 5.1.2]{ostrikremarks} show that aside from $\overline{\mathrm{Fib}}$ (up to equivalence), any spherical fusion category of prime global dimension $p=2,3,5,7,11$ is pointed.  The next smallest case is $p=13$.  Proposition \ref{dfdivides} implies that $d_{\mathrm{FPdim}(\mathcal{C})}\in\{1,2,3,6\}$.  The case $d_{\mathrm{FPdim}(\mathcal{C})}=1$ is pointed and the case $d_{\mathrm{FPdim}(\mathcal{C})}=6=(13-1)/2$ is impossible as in Theorem \ref{previous}.  If $d_{\mathrm{FPdim}(\mathcal{C})}=2$, then $\mathrm{FPdim}(\mathcal{C})\in\mathbb{Q}(\sqrt{13})$.  The fundamental unit $\epsilon_{13}:=(1/2)(3+\sqrt{13})$ has norm $-1$, thus any totally positive unit in $\mathbb{Q}(\sqrt{13})$ is $\epsilon_{13}^n$ for some even $n$.  But $\sigma(\epsilon_{13}^n)=\epsilon_{13}^{-n}$ and moreover $13\epsilon_{13}^{-n}<\sqrt{8/5}$ for all even $n\in\mathbb{Z}_{\geq2}$ thus no such number is the Frobenius-Perron dimension of a fusion category by \cite[Theorem 4.2.1]{ostrikremarks}.  Lastly assume $d_{\mathrm{FPdim}(\mathcal{C})}=3$.  Then  using the absolute trace bound of Flammang \cite{flammang}, $\mathrm{Tr}(u_{\mathrm{FPdim}(\mathcal{C})}^{-1})>3\cdot1.78>5$ as all exceptions to this bound do not lie in $\mathbb{Q}(\zeta_{13})^+$.  As the ring homomorphism $\dim$ must then have at least three Galois conjugates to satify \cite[Corollary 2.15]{ost15}, we must have $6\leq\mathrm{Tr}(u_{\mathrm{FPdim}(\mathcal{C})}^{-1})\leq10$ to satisfy Proposition \ref{notlapland}.  It is now a finite check to verify no such unit exists in $\mathbb{Q}(\zeta_{13})^+$ by evaluating $\sqrt{\Delta}$ where $\Delta$ is the discriminant of $x^3-ax^2+bx-1$ where $6\leq a\leq10$ and $b\in\mathbb{Z}_{\geq1}$ such that $\Delta>0$ (as $u_{\mathrm{FPdim}(\mathcal{C})}^{-1}$ is totally real).  One easily demonstrates this set is finite and if $u_{\mathrm{FPdim}(\mathcal{C})}^{-1}\in\mathbb{Q}(\zeta_{13})^+$ exists, then $\sqrt{\Delta}$ must be a power of $13$.  For example in the largest case, for $a=10$, $7\leq b\leq25$ and $\sqrt{\Delta}\in\mathbb{Z}$ if and only if $a=17$ when $\Delta=3^4$. 
\end{example}


\subsection{Results on spherical braided fusion categories}

\begin{lemma}\label{modlemma}
Let $\mathcal{C}$ be a spherical fusion category of prime global dimension $p\in\mathbb{Z}_{\geq2}$.  If $\mathcal{C}$ is braided, then $\mathcal{C}$ is nondegenerately braided, hence modular, or symmetric.
\end{lemma}

\begin{proof}
By Lemma \ref{plemma}, $\mathrm{FPdim}(\mathcal{C})=p\cdot u_{\mathrm{FPdim}(\mathcal{C})}$ for some $u_{\mathrm{FPdim}(\mathcal{C})}\in\mathcal{O}_{\mathbb{Q}(\zeta_p)}^\times$.  Hence $\mathrm{FPdim}(\mathcal{C}_\mathbb{Q})=1,p$ where $\mathcal{C}_\mathbb{Q}$ is the fusion subcategory consisting of objects of rational integer Frobenius-Perron dimension.  In the latter case, $\mathcal{C}_\mathbb{Q}$ is pointed, hence $\mathrm{rank}(\mathcal{C})\geq p$ and thus $\mathcal{C}$ itself is pointed.  Moreover $\mathrm{FPdim}(\mathcal{C})=p$ and $\mathcal{C}$ is equivalent to $\mathrm{Rep}(\mathbb{Z}/p\mathbb{Z})$ as a fusion category \cite[Corollary 8.30]{ENO}.  Braidings on $\mathrm{Rep}(\mathbb{Z}/p\mathbb{Z})$ are either symmetric or modular as $\mathrm{Rep}(\mathbb{Z}/p\mathbb{Z})$ has no proper nontrivial fusion subcategories.  Otherwise $\mathcal{C}_\mathbb{Q}$ is trivial, which contains the symmetric center of $\mathcal{C}$, i.e. \!$\mathcal{C}$ is modular.
\end{proof}

As spherical braided fusion categories of prime global dimension are nondegenerately braided, one can utilize the Galois action on the normalized modular data which will naturally involve traces of real cyclotomic integers.

\begin{lemma}\label{calegarilemma}
Let $p\in\mathbb{Z}_{>7}$ be prime.  If $\alpha\in\mathbb{Q}(\zeta_p)^+$ is a cyclotomic integer, then:
\begin{itemize}
\item[(a)] $\alpha=\pm1$ with $\mathrm{Tr}_{\mathbb{Q}(\zeta_p)^+}(\alpha^2)=(p-1)/2$,
\item[(b)] $\alpha$ is Galois conjugate to $\pm2\cos(2\pi/p)$ with $\mathrm{Tr}_{\mathbb{Q}(\zeta_p)^+}(\alpha^2)=p-2$, or
\item[(c)] $\mathrm{Tr}_{\mathbb{Q}(\zeta_p)^+}(\alpha^2)\geq p$.
\end{itemize}
\end{lemma}

\begin{proof}
This is implied by \cite[Lemma 3]{cassels} which states that if $\alpha$ is a (real) cyclotomic integer which is not a root of unity or the sum of two roots of unity, then $\mathrm{Tr}(\alpha^2)\geq2\cdot[\mathbb{Q}(\alpha^2):\mathbb{Q}]$.  Assuming $\alpha\in\mathbb{Q}(\zeta_p)^+$ with $p\neq2$, $\alpha$ being a root of unity or a sum of two roots of unity gives cases (a) and (b) (refer to \cite[Lemma 4.1.3]{MR2786219}, for example).  Otherwise, \cite[Lemma 3]{cassels} implies
\begin{align}
\mathrm{Tr}_{\mathbb{Q}(\zeta_p)^+}(\alpha^2)&=[\mathbb{Q}(\zeta_p)^+:\mathbb{Q}(\alpha^2)]\cdot\mathrm{Tr}(\alpha^2) \\
&\geq[\mathbb{Q}(\zeta_p)^+:\mathbb{Q}(\alpha^2)]\cdot2\cdot[\mathbb{Q}(\alpha^2):\mathbb{Q}] \\
&=2\cdot[\mathbb{Q}(\zeta_p)^+:\mathbb{Q}] \\
&=p.
\end{align}
\end{proof}

Simple objects of categorical dimension of type (b) in Lemma \ref{calegarilemma} rarely occur under the current assumptions.

\begin{lemma}\label{dimlemma}
Let $p\in\mathbb{Z}_{\geq7}$ be prime.  If $\mathcal{C}$ is a spherical braided fusion category of global dimension $p$, then $\dim(X)\neq\pm2\cos(2\pi/p)$ for any $X\in\mathcal{O}(\mathcal{C})$.
\end{lemma}

\begin{proof}
If $\dim(X)=\pm2\cos(2\pi/p)$ for some $X\in\mathcal{O}(\mathcal{C})$, then because $\mathcal{C}$ is modular, the Galois conjugates of $p(2\cos(2\pi/p))^{-2}$ are formal codegrees of $\mathcal{C}$ which are $(p-1)/2$ in number.  But this implies the categorical dimension representation $\dim:K(\mathcal{C})\to\mathbb{C}$ has at least $(p-1)/2$ Galois conjugates as well by \cite[Corollary 2.15]{ost15} and therefore $\mathcal{C}$ has formal codegree $p$ with multiplicity at least $(p-1)/2$.  These formal codegrees $f_Y$ correspond to simple objects $Y$ of squared categorical dimension $1$ via $f_Y=\dim(\mathcal{C})/\dim(Y)^2$, and together with the Galois orbit of $X$, these $p-1$ simple objects are all of $\mathcal{O}(\mathcal{C})$ as $\mathcal{C}$ cannot be pointed.  Moreover
\begin{equation}
p=\dim(\mathcal{C})=(p-1)/2+\mathrm{Tr}(\dim(X)^2)=(p-1)/2+p-2,
\end{equation}
which implies $p=5$.
\end{proof}

\begin{proposition}\label{theeprop}
Let $\mathcal{C}$ be a spherical braided fusion category of prime global dimension $p\in\mathbb{Z}_{\geq2}$.  If $p\neq5$, then $\mathcal{C}$ is pointed.
\end{proposition}

\begin{proof}
We may assume $p>7$ as the cases $p=2,3,7$ have already been shown to be pointed in \cite{ostrikremarks} and Theorem \ref{previous}.  Note that the unnormalized modular data of $\mathcal{C}$ is contained in $\mathbb{Q}(\zeta_{p^n})$ for some $n\in\mathbb{Z}_{\geq1}$ \cite[Theorem 3.9]{paul}.  In particular the (multiplicative) central charge $\tau^+/\sqrt{\dim(\mathcal{C})}$ is a $p^n$th root of unity; let $\gamma$ be any of its cube roots.  The normalized twists $t_X$ for $X\in\mathcal{O}(\mathcal{C})$ are now $\tilde{t}_X/\gamma$ where $\tilde{t}_X$ is a $p^n$th root of unity and $\sigma\in\mathrm{Gal}(\overline{\mathbb{Q}}/\mathbb{Q})$ acts on $t_X$ via $\sigma^2(t_X)=t_{\hat{\sigma}(X)}$ where $\hat{\sigma}:\mathcal{O}(\mathcal{C})\to\mathcal{O}(\mathcal{C})$ is the corresponding Galois permutation \cite[Theorem II (iii)]{dong2015congruence}.  As all normalized modular data lies in $\mathbb{Q}(\zeta_{3p^n})$ we need only study the action of
\begin{align}
G:=\mathrm{Gal}(\mathbb{Q}(\zeta_{3p^n})/\mathbb{Q})&\cong\mathrm{Gal}(\mathbb{Q}(\zeta_{3})/\mathbb{Q})\times\mathrm{Gal}(\mathbb{Q}(\zeta_{p^n})/\mathbb{Q}) \\
&\cong\mathbb{Z}/2\mathbb{Z}\times\mathrm{Gal}(\mathbb{Q}(\zeta_{p^n})/\mathbb{Q}).\label{torsion}
\end{align}
Let $H:=\{\sigma\in G:\sigma=\tau^2\text{ for some }\tau\in G\}$.  Because of the $2$-torsion in (\ref{torsion}),  $H$ can be identified with a subgroup of $\mathrm{Gal}(\mathbb{Q}(\zeta_{p^n})/\mathbb{Q})$ (which is cyclic).  As the subgroup of squares is index 2 when the order of a finite cyclic group is even and the subgroups of each index are unique, then $H\cong1\times\mathrm{Gal}(\mathbb{Q}(\zeta_{p^n})^+/\mathbb{Q})$ which we will identify with the nontrivial factor.

\vspace{3 mm}

If $\tilde{t}_X\neq1$ for some $X\in\mathcal{O}(\mathcal{C})$, then $\tilde{t}_X$ is a primitive $p^k$th root of unity for some $1\leq k\leq n$ and $\tilde{t}_X$ has $p^{k-1}(p-1)/2$ Galois $H$-conjugates.  Thus $k=1$ (or else $|
\mathcal{O}(\mathcal{C})|>p$) for all such $X$ and moreover $n=1$ above. This implies $t_{\hat{\sigma}(X)}\neq t_X$ for all $\sigma\in H$, hence the $H$-orbit ($\hat{\sigma}(X)$ for all $\sigma\in H$) of $X$ has no fixed points.  Moreover
\begin{equation}\label{but}
\dim(\mathcal{C})\geq\sum_{\sigma\in H}\dim(\hat{\sigma}(X))^2=\sum_{\sigma\in H}\sigma(\dim(X)^2)=\mathrm{Tr}_{\mathbb{Q}(\zeta_p)^+}(\dim(X)^2)
\end{equation}
Furthermore, by Lemmas \ref{calegarilemma} and \ref{dimlemma},  $\mathrm{Tr}_{\mathbb{Q}(\zeta_p)^+}(\dim(X)^2)\geq p$ unless $\dim(X)^2=1$ because $p$ was assumed to not be $5$.   Hence $\dim(X)^2=1$ otherwise all simple objects are in the $H$-orbit of $X$, and hence all have nontrivial twist (but the unit must have trivial twist).  An $X\in\mathcal{O}(\mathcal{C})$ with $\tilde{t}_X\neq1$ does exist, otherwise $\mathcal{C}$ is integral, hence pointed.

\vspace{2 mm}

In summary, $\mathcal{C}$ posesses (at least) $(p-1)/2$ nonisomorphic simple objects with nontrivial twists and squared categorical dimension 1.  If any other orbit of nontrivial twists exists, then their squared categorical dimensions are 1 as well, hence $\mathcal{C}$ is pointed (c.f. \cite[Exercise 9.6.2]{tcat}).  So finally we assume that the remainder of simple objects of $\mathcal{C}$ have trivial twists, and the sum of their squared categorical dimensions is necessarily $p-(p-1)/2=(p+1)/2$.  This is sufficient information to compute (in principle) the Gauss sums $\tau^{\pm}$ of $\mathcal{C}$, whose product is $\tau^+\cdot\tau^-=\dim(\mathcal{C})=p$ \cite[Proposition 8.15.4]{tcat}.  If we define $\alpha:=\sum_{\sigma\in H}\tilde{t}_{\hat{\sigma}(X)}$ and $\tau$ to be a (cyclic) generator of $\mathrm{Gal}(\mathbb{Q}(\zeta_p)/\mathbb{Q})$.  Then evidently $\alpha$ is either $\sum_{\text{odd }a}\tau^a(\zeta_p)$ or $\sum_{\text{even }a}\tau^a(\zeta_p)$ for $0<a<p$ as $\tilde{t}_X$ differs from $\tilde{t}_{\hat{\sigma}(X)}$ by an even power of $\tau$.  These are complex conjugates of one another, hence $\alpha+\overline{\alpha}=-1$ in either case, the sum of all nontrivial $p$th roots of unity where here $\overline{\alpha}$ is the complex conjugate of $\alpha$.  Moreover
\begin{equation}
4p=4\tau^+\tau^-=\left(p+1+2\alpha\right)\left(p+1+2\overline{\alpha}\right)=(p+1)^2-4(p+1)+4|\alpha|^2.
\end{equation}
Thus $4p-(p+1)(p-3)=4|\alpha|^2$. But the left-hand side is negative for $p>7$ and the right-hand side is always positive, hence no such fusion category exists.
\end{proof}

\begin{theorem}\label{biggun}
Let $p\in\mathbb{Z}_{\geq2}$ be prime.  If $\mathcal{C}$ is a spherical braided fusion category with $N(\dim(\mathcal{C}))=p$, then $\mathcal{C}$ is pointed, or equivalent to $\mathrm{Fib}$, $\mathrm{Fib}^\sigma$, or $\overline{\mathrm{Fib}}$ as a spherical fusion category.
\end{theorem}

\begin{proof}
This follows from Proposition \ref{theeprop}, \cite[Example 5.1.2(v)]{ostrikremarks}, and Proposition \ref{prop6}. 
\end{proof}


\section{Discussion}

Formal codegrees are necessary to the understanding of fusion categories, so here we discuss possible ways they may be inspected further.

\par A classification of modular tensor categories $\mathcal{C}$ such that the Galois action on $\mathcal{O}(\mathcal{C})$ is transitive was recently announced by Ng, Wang, and Zhang and consists of the categories $\boxtimes_{j=1}^n\mathcal{T}_{\kappa_j}^{\sigma_j}$ from Example \ref{sl2}.  We conjecture a result which extends to fusion categories which are not necessarily modular.

\begin{conjecture}
Let $\mathcal{C}$ be a fusion category.  If $\mathcal{C}$ has a formal codegree $f$ such that for all $m\in\mathbb{Z}_{\geq2}$, $m$ does not divide $f$, then $\mathcal{C}$ is equivalent to a fusion category of the form $\boxtimes_{j=1}^n\mathcal{T}_{\kappa_j}^{\sigma_j}$.
\end{conjecture}

In Proposition \ref{bigone}, we were able to show that for each formal codegree $f$ of a spherical fusion category, there exists a prime dividing the norm of $f$ with a large exponent relative to the degree of the extension $\mathbb{Q}(f)/\mathbb{Q}$.  But it is likely that more is true from an inspection of known examples.

\begin{question}
Does there exist a spherical fusion category $\mathcal{C}$ with formal codegree $f$ whose norm has prime factorization $N(f)=\prod_{j\in J}p_j^{k_j}$, and $k_j<\frac{1}{2}d_f$ for some $j\in J$?
\end{question}

\par One could also extend the results of Theorem \ref{prop6} by studying formal codegrees $f$ of spherical fusion categories such that $N(f)=\prod_{j\in J}p_j^{k_j}$ with $0\leq k_j\leq2$.  Such a formal codegree has $[\mathbb{Q}(f):\mathbb{Q}]\leq4$ so there is a good chance these may be analyzed by elementary methods such as those found in Example \ref{thurteen}.

\par Furthermore, norm finiteness for the global dimension of fusion categories (or more trivially Frobenius-Perron dimension) may be the shadow of a more general phenomenon.  In particular, these results heavily relied on the fact that there exist finitely many fusion categories (up to equivalence) with a fixed global/Frobenius-Perron dimension.  Compare this to the near-group fusion categories $\mathcal{N}(G,k)$ (Example \ref{neargroup}).  It is expected that if $G$ is a finite group, the near-group fusion categories $\mathcal{N}(G,k)$ exist for only finitely-many $k\in\mathbb{Z}_{\geq1}$.  One peculiarity of existence for infinitely-many $k\in\mathbb{Z}_{\geq1}$ is that this would be an infinite collection of inequivalent fusion categories sharing the same formal codegree: $|G|$.  This motivates the following question.

\begin{question}
Let $f\in\mathbb{C}$.  Does there exist an infinite family of nonisomorphic categorifiable fusion rings each having $f$ as a formal codegree?
\end{question}

\par The proof of Theorem \ref{biggun} was made possible by the theoretical framework of modular tensor categories.  But Theorem \ref{previous} is significant evidence that the assumption of a braiding may be superfluous.  Regardless, the assumption of sphericality may be more difficult to remove.  If $\dim(\mathcal{C})=p$ on the nose, then one can classify fusion categories of global dimension $2p$, which includes the sphericalization $\tilde{\mathcal{C}}$, and work backwards.  But when considering the norm of global dimension the problem becomes significantly more difficult.  In particular the degree $d_{\dim(\mathcal{C})}=[\mathbb{Q}(\dim(\mathcal{C})):\mathbb{Q}]$ is bounded above by $k:=\lfloor\log(p)/\log(4/3)\rfloor$ (refer to the proof of Theorem \ref{normfiniteness}) and $N(\dim(\tilde{\mathcal{C}}))=2^{d_{\dim(\mathcal{C})}}N(\dim(\mathcal{C}))$.  Classifying spherical fusion categories whose global dimension has large prime factors seems unlikely at this time, so another method may be needed.  Yet, we still present the following conjecture as a goal.

\begin{conjecture}
Let $p\in\mathbb{Z}_{\geq2}$ be prime.  If $\mathcal{C}$ is a fusion category whose global dimension has norm $p$, then $\mathcal{C}$ is pointed, or equivalent to $\mathrm{Fib}$, $\mathrm{Fib}^\sigma$, or $\overline{\mathrm{Fib}}$.
\end{conjecture}

\begin{ack}
We would like to thank Terry Gannon and Victor Ostrik for productive discussions during the preparation of this manuscript, and we would like to thank Colleen Delaney and Julia Plavnik for reading an early draft.
\end{ack}

\bibliographystyle{plain}
\bibliography{bib}

\end{document}